\documentclass[a4paper,english,oneside]{amsart}
\usepackage[utf8]{inputenc}
\usepackage[T1]{fontenc}
\usepackage[hscale=0.6, vscale=0.66]{geometry}

\usepackage{amssymb}

\usepackage{mathtools}

\DeclareFontFamily{OMX}{mlmex}{}
\DeclareFontShape{OMX}{mlmex}{m}{n}{%
   <->mlmex10%
   }{}%
\usepackage{mlmodern}

\theoremstyle{plain}
\newtheorem{theorem}{Theorem}
\newtheorem{proposition}[theorem]{Proposition}
\newtheorem{lemma}[theorem]{Lemma}

\theoremstyle{definition}

\newtheorem{remark}[theorem]{Remark}

\usepackage{color}

\usepackage[np,autolanguage]{numprint}
\AtBeginDocument{\npthousandthpartsep{\,}}

\usepackage{hyperref}
\hypersetup{%
  colorlinks=true,%
  citecolor=[RGB]{120,29,126},
  pdfauthor={Jean-François Burnol},%
  pdfsubject={Baillie series},%
  pdfstartview=FitH,%
  pdfpagemode=UseNone,%
}

\newcommand\plusbas{\vphantom{X^X}}

\newcommand\cX[1][]{%
  \if\relax\detokenize{#1}\relax\else{}_{\plusbas#1}\mskip-1mu\relax\fi
  \mathcal{X}%
}
\newcommand\cY[1][]{%
  \if\relax\detokenize{#1}\relax\else{}_{\plusbas#1}\mskip-1mu\relax\fi
  \mathcal{Y}%
}

\newcommand\sD[1][]{%
  \if\relax\detokenize{#1}\relax\else{}_{\plusbas#1}\mskip-1mu\relax\fi
  \mathsf{D}%
}
\newcommand\bsD[1][]{%
  \if\relax\detokenize{#1}\relax\else{}_{\plusbas#1}\mskip-1mu\relax\fi
  \boldsymbol{\mathsf{D}}%
}

\newcommand\NN{\mathbb{N}}
\newcommand\ZZ{\mathbb{Z}}

\newcommand\ld{\mathrm{ld}}
\newcommand\sLD{\mathsf{LD}}

\newcommand\dx{\mathrm{d}\mskip-1mu x}

\newcommand\muld[1]{{}_{\plusbas#1}\mskip-1mu\mu}
\newcommand\nuld[1]{{}_{\plusbas#1}\mskip-1mu\nu}

\newcommand\rnot[1]{(\setminus #1)}

\allowdisplaybreaks

\begin{document}

\title[Exactly one 42]{%
  Summing the  ``exactly
  one $42$'' and similar subsums of the harmonic series%
}

\author[J.-F. Burnol]{Jean-François Burnol}

\address{Université de Lille,
  Faculté des Sciences et technologies,
  Département de mathématiques,
  Cité Scientifique,
  F-59655 Villeneuve d'Ascq cedex,
  France}
\email{jean-francois.burnol@univ-lille.fr}
\date{February 22, 2024}



\begin{abstract}
  For $b>1$ and $\alpha\beta$ a string of two digits in base $b$, let $K_1$ be
  the subsum of the harmonic series with only those integers having exactly
  one occurrence of $\alpha\beta$.  We obtain a theoretical representation of
  such $K_1$ series which, say for $b=10$, allows computing them all to
  thousands of digits.  This is based on certain specific measures on the unit
  interval and the use of their Stieltjes transforms at negative integers.
  Integral identities of a combinatorial nature both explain the relation to
  the $K_1$ sums and lead to recurrence formulas for the measure moments
  allowing in the end the straightforward numerical implementation.
\end{abstract}

\maketitle

\section{Introduction}

We obtain a  numerically
efficient theoretical representation of the sums
\[ K_1 = \sum\nolimits'_{n>0} \frac1n\;,\]
where the prime symbol means that the contributing integers are those which
contain exactly \emph{one} instance of a given two-digit string $\alpha\beta$
of base $b$ digits.  More generally one defines sums $K_i$ where the integers
contributing are those having exactly $i$ occurrences of a given string
$d_1\dots d_k$.  For $i=0$, and all $k\geq1$, they were considered by
Schmelzer and Baillie \cite{schmelzerbaillie} who gave a numerical algorithm.
For $i>0$ no such algorithm was known, as stressed in \cite[Section
12]{bailliearxiv}, except for $k=1$, then the $K_i$'s are special cases of
Irwin series \cite{irwin}, for which Baillie gives in \cite{bailliearxiv} an
algorithm with implementation.  Here we consider two-digits strings
($k=2$) and obtain efficient exact formulas for the $i=0$
and $i=1$ cases, extending on our earlier work
\cite{burnolkempner,burnolirwin} which treated in a novel manner all
$i$'s for $k=1$.

The convergence of the $K_1$ series is (almost always) extremely slow: one can
estimate (from equations \eqref{eq:Zdd}, \eqref{eq:Zde}) that after having
added the contribution of all integers with at most $\ell$ digits (in base
$b$), the next block of integers with $\ell+1$ digits will (asymptotically)
reduce the distance to the limit by only about one part in $b^2$.  This is
only a heuristic based on the exact count of kept integers of a given length,
and the heuristic result actually is a bit better for $\alpha\neq\beta$ and a
bit worse for $\alpha=\beta$: for $b=2$ one expects that integrating the
$(l+1)$-block multiplies the remaining error by about $1/2$ in the
$\alpha\neq\beta$ case and by $(1+\sqrt5)/4\approx\np{0.809}$ in the
$\alpha=\beta$ case.

For $b=10$ for example, this means that evaluating (not even precisely!) by
``brute force'' is impossible.  For $b=2$ there are four cases $00$, $01$,
$10$, and $11$.  The $01$ and $10$ cases are, as will be discussed next,
easily accessible to direct computation.  Not so for $00$ and $11$: from the
$\approx\np{0.809}$ ratio we expect to have to extend the summation range
$2^{11}$-fold to gain only one decimal figure.  We computed the partial
sums with integers having up to 36 binary digits, they are in both cases at less
than $\np{0.02}$ from the full $K_1$'s.  In itself this required some
combinatorics developing further what will be touched upon in the first
section of this article.
\begin{itemize}
\item for $b=2$, $\alpha\beta=00$, the $K_1$ sum will round to either
  $\np{2.76}$ or $\np{2.77}$.  This is obtained from summing the first
  $\np{355817324}$ terms and adding a correction.
\item for $b=2$, $\alpha\beta=11$, the $K_1$ sum will round to either
  $\np{2.93}$ or $\np{2.94}$.  This is obtained from summing the first
  $\np{237387960}$ terms and adding a correction.
\end{itemize}
This meager data is \emph{all that we have} to gain confidence in any
numerical implementation of Theorem \ref{thm:K1dd}!  Here is what it says, via
the author \textsc{SageMath} code (values are rounded):
\begin{itemize}
\item $b=2$, $\alpha\beta=00$: $K_1\approx\np{2.7633212517890266334181008765657}$.\hfill
\item $b=2$, $\alpha\beta=11$: $K_1\approx\np{2.9384134076501894515174038829017}$.\hfill
\end{itemize}
Hundreds, even thousands of digits can be obtained.  But we have nothing
pre-existing to compare to, already for the third decimal!

Still for base $2$: for $\alpha\beta=10$, one convinces oneself easily that
\[
K_1 = \sum_{n\geq2}\sum_{0\leq j_1<j_2<n}\frac1{2^n-2^{j_2}+2^{j_1}-1}
\]
This can be evaluated numerically directly and has value about
$\np{3.013662769896}$.
Here is what Theorem \ref{thm:K1} says in that case:
\begin{equation*}
K_1 = 4(\frac12+\frac13)
  +\sum_{m=1}^\infty (-1)^m v_m \left(\frac{1}{3^{m+1}}-\frac1{6^{m+1}}\right)
  + \sum_{m=1}^\infty (-1)^m u_m\left(\frac1{2^{m+1}} + \frac1{6^{m+1}}\right)
\end{equation*}
where $(u_m)$ and $(v_m)$ start with $u_0=v_0=4$ and obey the following
recurrences:
\begin{align*}
(2^{m+1}-2+2^{-m-1}) u_m &= \sum_{j=1}^{m}\binom{m}{j}(1 - 2^{-m-1+j})u_{m-j}\\
(2^{m+1}-2+2^{-m-1}) v_m &=
\begin{aligned}[t]
  &\sum_{j=1}^{m}\binom{m}{j} \Bigl((1 - 2^{-m-1+j})v_{m-j} + 2^{-m-1+j}
  u_{m-j}\Bigr) \\&+ 2^{-m-1}u_m
\end{aligned}
\end{align*}
We checked that this formula is numerically correct up to hundreds
of digits.

For $\alpha=01$, the series can be written as:
\[
K_1 = 2 \sum_{n\geq3}\sum_{0< j_1<j_2<n}\frac1{2^n-2^{j_2}+2^{j_1}-1}\approx\np{2.813935234961}
\]
We checked to hundreds of decimal digits that Theorem \ref{thm:K1}
does compute this correctly.

With the results of this paper, we switch from not knowing a single decimal
figure to being able to do statistics.  A simple heuristics from Baillie
\cite[Section 12]{bailliearxiv} suggests $b^2\log(b)$ as reasonable estimate.
From the above results, we see it is only about ok for $b=2$ but when $b$
increases it becomes more and more accurate.  Indeed, for $b=10$, the maximal
distance from $100\log(10)$ to the one hundred $K_1$'s is about $\np{0.026}$.
It is even less than $\np{0.0012}$ if $\alpha\neq\beta$.  Here is for example
$K_1$ for $b=10$, $\alpha\beta=42$ (the values are rounded):
\begin{verbatim}
   K1 = 230.25882 13214 33508 40478 77627 59267 85873 95858 57341
57966 44057 49270 12717 79357 21101 87579 14837 04726 00875 18443
\end{verbatim}
And here are a few more
\begin{verbatim}
K1(10;"35") = 230.25886 98636 06045 19996 74060 14171 11117 72617
37929 58025 46555 85173 41436 96256 50505 60878 23149 34317 80138
K1(10;"00") = 230.25778 86509 07954 96301 56932 54264 57777 94887
61390 55856 69063 45994 61450 04414 55054 51076 76032 63035 62661
K1(10;"99") = 230.25941 83393 23881 16119 42823 97032 86550 26076
61541 81786 98122 41289 02370 29019 23693 24903 93542 26528 80528
\end{verbatim}
Compare with $100\log(10)\approx\np{230.2585092994}$.  For single-digit
strings (not two-digits as considered here), Fahri theorem \cite{fahri} says
that $K_1 > b \log(b)$.  We see here that $K_1 > b^2 \log(b)$ does not always
hold: for $b=10$, there are four $K_1$ values below $100\log(10)$, those for
$\alpha=\beta\in\{0,1,2,3\}$.

The numerical evaluations mentioned here are based on theoretically
\emph{exact} formulas.  As a further example of such formulas consider the
case $b=2$, $\alpha=1$, $\beta=0$ for which $K_0$ is the Erdös-Borwein
constant $E = \sum_{n=1}^\infty(2^n-1)^{-1}$.  Its numerical value is
$\approx{\np{1.60669515241529}}$.  Using the same sequence $(u_m)$ which
intervened in the computation of $K_1$ for $b=2$ and
$\alpha\beta=10$, Theorem \ref{thm:K0} says:
\begin{equation*}
  E  = 1 + 4(\frac13-\frac16) + 
                          \sum_{m=1}^\infty(-1)^m u_m(\frac1{3^{m+1}}-\frac1{6^{m+1}})
\end{equation*}
We checked the numerical validity up to one thousand digits.  In this case
using the defining series is more efficient (for large $m$) due to the
quadratic cost in evaluating the sequence $(u_m)$.  Besides, there is the very
hard to beat formula
\begin{equation*}
E = \sum_{n=1}^\infty 2^{-n^2}\frac{1+2^{-n}}{1-2^{-n}}
\end{equation*}
(Clausen identity  \cite[exercise 10 on page 98]{borweinborwein} with
$q=\frac12$).  It would be
wonderful if such quickly converging series could emerge more generally in
this area.

Let's now detail the contents of the paper.  The value predicted as
``reasonable'' by the Baillie heuristic argument relies on the knowledge of
how many integers of a given length verify the condition imposed upon them.
This combinatorial exercice is solved here in the first section, which after
defining measures on \emph{strings} needed to make sure they had a finite
total mass.  The Remark \ref{rem:gf} and the text which follows it give some
details on how this is relevant to the heuristic estimate.  The combinatorics
underlying the determination of the mass generating functions gets amplified
throughout the paper.  It gets converted into integral identities, after
having pushed the measures from strings to $b$-imal numbers in the half-open
unit interval $[0,1)$.  As in \cite{burnolkempner} and \cite{burnolirwin} the
starting points are some ``log-like'' expressions:
\[
K_0 = \int_{[b^{-1},1)}\frac{d\mu(x)}{x}\qquad K_1 = \int_{[b^{-1},1)}\frac{d\nu(x)}{x}
\]
The combinatorial properties allow to reformulate these integrals as sums of
integrals on the full half-open unit interval, which are actually (up to sign)
the Stieltjes transforms evaluated at certain negative integers.  There is a
subtlety here that in the $\alpha=\beta$ case we need to switch attention to
two other measures $\sigma$ and $\tau$.

The Stieltjes transforms obey recursive equations which would allow, up to
some proliferation of terms, to express $K_0$ and $K_1$ as geometrically
converging series with arbitrarily small ratios.  We content ourself to do
only two steps and to reach ``level $2$'' (the above log-like formulas being
``level $0$'').  Our concluding Theorems \ref{thm:K0}, \ref{thm:K1}, \ref{thm:K0dd}
and \ref{thm:K1dd} say that the numerical evaluation of the $K_0$ and $K_1$
sums is possible via combining the measure moments with finite sums of
reciprocal powers of certain integers having from $2$ to $4$ digits in base
$b$ (only up to $3$ in the $K_0$ case). It is possible to compute numerically
the moments because they verify a linear recurrence which is another
manifestation of the underlying pervading combinatorics.

The multi-precision values given in this introduction were obtained from a
\textsc{SageMath} script which, together with a variant using \textsc{Python}
double-precision floats, is available from the author.  With it the $K_0$ and
$K_1$ sums can be computed to thousands of digits (small bases $b$ require more
terms and would benefit from going to ``level $3$'').  Convergence ratio is
never worse than $1/b\leq 1/2$; ``level $3$'' formulas would achieve at least
$1/b^2$ geometric convergence but have more contributions.


\section{Strings}

As in \cite{burnolirwin} we call $\sD = \{0,\dots,b-1\}$ the set of ``digits''.

The space of \emph{strings} $\boldsymbol{\sD}$ is defined to be the union of
all cartesian products $\sD^l$ for $l\geq0$.  Notice that for $l=0$ one has
$\sD^0=\{\emptyset\}$ i.e. the set with one element.  We will call this
special element the \emph{none-string}.  The length $|X|=l$ of a string $X$ is
the integer $l$ such that $X\in\sD^l$.  We  may write $\bsD_l$ in
place of $\sD^l$.

There is a map from strings to
integers which to $X=(d_{l},\dots ,d_1)$ assigns $n(X) = d_{l} b^{l} + \dots +
d_1$ (and which sends the none-string to zero).  This map is many-to-one, and
each integer $n$ has a unique minimal length representation $X(n)$ which is
called \emph{the} $b$-representation of $n$.  The other representations differ
from this one only via leading zeroes.  Notice that the none-string and
strings with only zeroes all are representations of $0$, and $X(0)$ is the
none-string, not the one-digit string ``$0$''.

The length of an integer $n$ is $l(n) = |X(n)|$. Thus, $l(n)$ is the smallest
non-negative integer such that $n<b^{l(n)}$. Beware that $l(0)=0$, not $1$.

For any string of positive length we let $\sLD(X)$ be its leading digit, as
an element of $\sD$. For any digit $d$ we let $\bsD[d]$ be the subset of
$\bsD$ of strings having the digit $d$ as leading digit.  We may also use the
notation $\bsD[d]_l$ for $\bsD[d]\cap \bsD_l$.

We now let $\cX\subset \bsD$ be the set of strings not containing
$\alpha\beta$ as substring.  The notations $\cX[d]$, $\cX_l$ and $\cX[d]_l$
are self-explanatory.  We may also use $\cX[\rnot d]$ to mean the complement
in $\cX$ of $\cX[d]$.  This is almost the same as the union of the
$\cX[e]$ for $e\in\sD\setminus\{d\}$ except that it contains the none-string.
Other notations were considered such as striking through $d$ but it was feared
they could cause issues in HTML conversion.

We let $\cY\subset \bsD$ be the set of strings containing exactly one
occurrence of $\alpha\beta$.   The notations $\cY[d]$, $\cY_l$ and $\cY[d]_l$,
and others, 
are self-explanatory.

The \emph{pruning} map $\tau:\bsD\setminus \sD^0\to\bsD$ is the process
of removing the leading digit of a string.  Any string $X$ has exactly $b$
antecedents in the full string space $\bsD$ under $\tau$.  Their set is the
\emph{fiber} over $X$.  We may use fiber to refer to intersection with some
subsets such as $\cX$ or $\cY$, this should be clear from context.

\section{Measures and generating functions of masses}

We define a measure $\mu$ on $\bsD$ (equipped with the full $\sigma$-algebra
of all subsets) by assigning the weight $b^{-|X|}$ to any string $X$ in $\cX$,
and the zero weight to all other strings.  We shall prove later that this is a
finite measure having $b^2$ as its total mass.  For this we will use a
disjoint decomposition
\begin{equation}
  \mu = \delta_{\emptyset}+ \sum_{d\in\sD} \muld{d}
\end{equation}
where $\muld{d}$ is defined exactly as $\mu$ except that its support is
restricted to those strings in $\cX$ having $d$ as leading digit.  Notice that
we needed the  Dirac at the none-string for the above formula to be valid.  We
shall use the notation $\muld{\rnot{d}}$ to refer to the measure which has the
same weights as $\mu$, but only on $\cX[\rnot{d}]$.  That is:
\begin{equation}
  \muld{\rnot{d}} = \delta_{\emptyset}+ \sum_{e\neq d} \muld{e} = \mu - \muld{d}
\end{equation}

We define a measure $\nu$ on $\bsD$ by assigning the weight $b^{-|Y|}$ to any
string $Y$ in $\cY$, and the zero weight to all other strings.  We shall prove
later that $\nu$ is a finite measure, also having $b^2$ as its total mass.
We shall use a disjoint decomposition
\begin{equation}
  \nu = \sum_{d\in\sD} \nuld{d}
\end{equation}
where $\nuld{d}$ is defined exactly as $\mu$ except that its support is
restricted to those strings in $\cY$ having $d$ as leading digit.  Here the
none-string does not intervene.  We also use the notation $\nuld{\rnot{d}}$
which here means:
\begin{equation}
  \nuld{\rnot{d}} =  \sum_{e\neq d} \nuld{e} = \nu - \nuld{d}
\end{equation}

The next lemma is the main combinatorial fact which is the key to our solution
to the Baillie problem, in combination with techniques we introduced in our
earlier papers \cite{burnolkempner}, \cite{burnolirwin}.
\begin{lemma}\label{lem:main}
  The push-forward measures $\tau_*(\muld{d})$, $\tau_*(\nuld{d})$, $d\in\sD$ verify:
  \begin{align*}
    \tau_*(\muld{d}) &=\frac1b
    \begin{cases}
      \mu&(d\neq \alpha)\\
      \muld{\rnot\beta}&(d=\alpha)
    \end{cases}\\
    \tau_*(\nuld{d}) &= \frac1b
    \begin{cases}
      \nu&(d\neq \alpha)\\
      \nuld{\rnot\beta} + \muld{\beta} &(d=\alpha)
    \end{cases}
  \end{align*}
\end{lemma}
\begin{remark}
  In defining (in the expected way) the push-forward of a measure under the
  pruning map, we first restrict the measure to the complement of $\sD^0$,
  because the pruning map $\tau$ is not defined on $\sD^0=\{\textit{the
    none-string}\}$.  So we must be careful that if the measure assigns some
  mass to the none-string, the total mass of the push-forward will be lower
  than the original total mass.  But it will always be true that the mass of
  any $\sD^l$ for $l\geq1$ is the same as the mass given to $\sD^{l-1}$ by the
  push-forward measure.
\end{remark}
\begin{proof}
  Consider the strings $X\in\cX[d]_{l+1}$ with $d\neq\alpha$.  Pruning the
  first digit, we certainly get a string in $\cX_l$, i.e.\@ a string of length
  $l$ not containing the string $\alpha\beta$.  There is only one antecedent
  in the support of $\cX[d]$.  If $d=\alpha$, pruning the first digit also
  gives us a string in $\cX_l$ but its leading digit will not be $\beta$.

  Conversely suppose that $X'$ is a given string of length $l$ and
  not containing the string $\alpha\beta$.  If $d\neq\alpha$ there is one and
  exactly one string $X$ in the support of $\muld{d}$ whose pruning in $X'$.
  If $d=\alpha$ also, but under the condition that $X'$ does not have
  $\beta$ as leading digit (thus $X'$ can be the none-string).

  Let's check that we did not make a reasoning mistake for $X'$ the
  none-string.  It has indeed exactly one antecedent in the support of
  $\muld{d}$, it is the one-digit string $d$.  This works whether or not
  $d=\alpha$. So the formula in the Lemma is true as our definition of
  $\muld{\rnot\beta}$ was the restriction of $\mu$ to the set of strings not
  having $\beta$ as leading digit.  This includes the none-string.  If we had
  defined $\muld{\rnot\beta}$ as the sum of the $\muld{e}$ for $e\neq \beta$ we
  would have had to add an explicit $\delta_{\emptyset}$ somewhere.

  Let now $X$ be a string in $\cY[d]_{l+1}$.  Suppose first
  $d\neq\alpha$. Pruning $d$ we get a string which still has exactly one
  occurrence of $\alpha\beta$.  Conversely if we have such a string $X'$ we
  can add $d$ as leading digit and it gives an $X$ containing exactly one
  $\alpha\beta$.
  Suppose now $d=\alpha$.  Then pruning it we either get a string starting
  with $\beta$ and not containing any occurrence of $\alpha\beta$ or we get a
  string not starting with $\beta$ and containing exactly one occurrence of
  $\alpha\beta$.  Conversely if such a string $X'$ is given then $\alpha X'$
  will contain exactly one occurrence of $\alpha\beta$.

  The proof is complete.
\end{proof}

Define for any $l\in\NN$ (throughout $\NN=\ZZ_{\geq0}$), $M(l,d)$ to be
$\mu(\bsD[d]_l)$ and $N(l,d)$ to be $\nu(\bsD[d]_l)$.  We have $M(0,d) = 0 =
N(0,d)$ for all $d\in\sD$.  We get from the previous Lemma, for $l\geq0$ (the
parenthesized sums are $\mu(\bsD_l)$ and $\mu(\bsD_l)-\mu_\beta(\bsD_l)$
respectively):
\begin{equation}
  M(l+1,d) = \tau_*(\muld{d})(\bsD_l) =
  \begin{cases}
    \frac1b \left(\sum_{e\in\sD} M(l,e) + \delta_{l=0} (l)\right)&(d\neq\alpha)\\
    \frac1b \left(\sum_{e\in\sD, e\neq \beta} M(l,e) + \delta_{l=0}(l)\right) &(d=\alpha)
  \end{cases}
\end{equation}
Let us now consider the generating series $W_d = \sum_{l\geq0} M(l,d)T^l$ If
$d\neq\alpha$ we obtain (as there is no constant term):
\begin{equation}
  W_d = \sum_{l\geq0} M(l+1,d)T^{l+1} = \frac Tb\left(\sum_{e\in\sD} W_e + 1\right)
\end{equation}
Let $W = 1 + \sum_{e\in\sD} W_e = \sum_{l\geq0} \mu(\sD^l)T^l$.  We thus get
that all $W_d$ for $d\neq \alpha$ are equal to $Tb^{-1} W$.  For $d=\alpha$ we
get:
\begin{equation}\label{eq:walpha1}
  W_\alpha = \sum_{l\geq0} M(l+1,\alpha)T^{l+1}=\frac
  Tb\left(\sum_{e\in\sD\setminus{\beta}} W_e + 1\right) = \frac Tb(W - W_\beta)
\end{equation}
This implies from $W = 1 + \sum_e W_e$:
\begin{equation}
  W = 1 + (b-1)Tb^{-1} W  + Tb^{-1}(W - W_\beta) = 1 + TW - b^{-1}T W_\beta
\end{equation}
hence
\begin{equation}\label{eq:wbeta}
  W_\beta = b \frac{1 - W + TW}{T}
\end{equation}
and then from \eqref{eq:walpha1}:
\begin{equation}\label{eq:walpha2}
  W_\alpha = W - 1 - (1 - b^{-1})TW
\end{equation}
Consider first $\alpha=\beta$.  Combining the last two equations
\eqref{eq:wbeta} and \eqref{eq:walpha2} we get
$
   T W- T - (1 - b^{-1})T^2 W = b - bW + bTW
$
and
\begin{equation}\label{eq:Wdd}
  \begin{aligned}
    W &= \frac{b(b+T)}{b^2-b(b-1)T- (b-1)T^2}\\
    W_\alpha &= \frac{b T}{b^2-b(b-1)T- (b-1)T^2}\\
    W_d &= \frac{(b+T)T}{b^2-b(b-1)T- (b-1)T^2}\mathrlap{\quad(d\neq\alpha)}
  \end{aligned}
\end{equation}
Suppose now that $\alpha\neq\beta$.  In that case we must have $W_\beta = T b
^{-1} W$ and combining with \eqref{eq:wbeta} we obtain:
\begin{equation}\label{eq:Wde}
  \begin{aligned}
    W &= \frac{b^2}{b^2-b^2T + T^2}\\
    W_\alpha &= \frac{(b-T)T}{b^2-b^2T+T^2}\\
    W_d &= \frac{bT}{b^2-b^2T + T^2}\mathrlap{\quad(d\neq\alpha)}
  \end{aligned}
\end{equation}
So we know the total mass:
\begin{proposition}\label{prop:mumass}
  The measure $\mu$ has a finite total mass.  If $\alpha=\beta$ the
  generating functions for the masses of the $\sD^l$ are given by
  \eqref{eq:Wdd} and as corollary:
\begin{equation}
   \mu(\bsD) = b(b+1)\qquad \mu(\bsD[d])=
   \begin{cases}
     b& (d=\alpha)\\
     b+1&(d\neq\alpha)
   \end{cases}
\end{equation}
  If $\alpha\neq\beta$, the generating functions are given by \eqref{eq:Wde}
  and:
  \begin{equation}
   \mu(\bsD) = \mathrlap{b^2}\phantom{b(b+1)}\qquad \mu(\bsD[d])=
   \begin{cases}
     b-1& (d=\alpha)\\
     b &(d\neq\alpha)
   \end{cases}
  \end{equation}
\end{proposition}
\begin{proof}
  The values given are obtained via $T=1$ in the respective generating series.
\end{proof}

We now compute the generating series for the $N(l,d) = \nu(\bsD[d]_l)$,
$l\geq0$. They each start with a multiple of $T^2$.  Again
using Lemma \ref{lem:main} we have for $l\geq0$:
\begin{equation}
  N(l+1,d) = \tau_*(\nuld{d})(\bsD_l) =
  \begin{cases}
    \frac1b \sum_{e\in\sD} N(l,e)&(d\neq\alpha)\\
    \frac1b \left(\sum_{e\in\sD\setminus{\beta}} N(l,e) + M(l,\beta)\right) &(d=\alpha)
  \end{cases}
\end{equation}
Notice for example that for $l=1$ we correctly obtain a
$b^{-2}$ from the one-digit string $\beta$ which is image under pruning of
$\alpha\beta$ hence receives via the push-forward its $b^{-2}$ weight.
Let us now consider the generating series $Z_d = \sum_{l\geq1} N(l,d)T^l$.  If
$d\neq\alpha$ we obtain
\begin{equation}
  Z_d = \sum_{l\geq0} N(l+1,d)T^{l+1} = \frac Tb\sum_{e\in\sD} Z_e
\end{equation}
Let $Z = \sum_{e\in\sD} Z_e = \sum_{l\geq0} \nu(\sD^l)T^l$.    We obtain
that all $Z_d$ for $d\neq \alpha$ are equal to $Tb^{-1} Z$.  For $d=\alpha$ we
get:
\begin{equation}\label{eq:zalpha1}
  Z_\alpha = \sum_{l\geq0} N(l+1,\alpha)T^{l+1}=\frac
  Tb\left(\sum_{e\in\sD\setminus{\beta}} Z_e + W_\beta\right) = \frac Tb(Z -
  Z_\beta + W_\beta)
\end{equation}
Combining this with $Z_d =Tb^{-1} Z$ for $d\neq\alpha$ we get from $Z = \sum_e
Z_e$:
\begin{equation}
  Z = (1 - b^{-1})T Z  + \frac Tb(Z - Z_\beta + W_\beta) = 
   TZ - b^{-1} T Z_\beta + b^{-1} T W_\beta
\end{equation}
hence
\begin{equation}
  \label{eq:zbeta}
  Z_\beta = W_\beta + b (T-1)\frac ZT
\end{equation}
and then from \eqref{eq:zalpha1}:
\begin{equation}\label{eq:zalpha2}
  Z_\alpha = b^{-1}T(Z -  Z_\beta + W_\beta) = (1 - (1 - b^{-1})T)Z 
\end{equation}
Coonsider first $\alpha=\beta$.  Combining the last two equations
\eqref{eq:zbeta} and \eqref{eq:zalpha2} and
\eqref{eq:Wdd} for $W_\beta = W_\alpha$ we get then $Z$.  It turns out to be
$W_\beta^2$! We then also get $Z_\alpha$ and $Z_d$ for $d\neq\alpha$:
\begin{equation}
  \label{eq:Zdd}
  \begin{aligned}
    Z &= \frac{b^2 T^2}{(b^2-b(b-1)T- (b-1)T^2)^2}\\
    Z_\alpha &= \frac{(b^2 - b(b-1)T)T^2}{(b^2-b(b-1)T- (b-1)T^2)^2}\\
    Z_d &= \frac{b T^3}{(b^2-b(b-1)T- (b-1)T^2)^2}\mathrlap{\quad(d\neq\alpha)}
   \end{aligned}
\end{equation}

Suppose now that $\alpha\neq\beta$.  In that case we must have $Z_\beta = T b
^{-1} Z$ and combining with \eqref{eq:zbeta} we obtain:
\begin{equation}
  Z = \frac{bT W_\beta}{b^2 - b^2T + T^2} = W_\beta^2
\end{equation}
so again $Z = W_\beta^2$! A posteriori the combinatorial meaning is obvious:
\[ \underbrace{xxxxxxxx\alpha}_{\text{reverse this}}\beta xxxxxxxx\]
  When reversing the order of digits in the first half we get a string not
  containing $\beta\alpha$ and starting with $\alpha$.  But obviously the
  number of such strings depend only on the length and on whether
  $\alpha=\beta$ or $\alpha\neq \beta$.  This situation does not depend on
  whether we work with $\alpha\beta$ or with $\beta\alpha$.  So we have as
  many possibilities as for strings of the same length, not containing
  $\alpha\beta$ and starting with $\beta$.  So $Z = W_\beta^2$ from Cauchy
  product.

Summarizing this case $\alpha\neq\beta$ we thus have:
\begin{equation}\label{eq:Zde}
  \begin{aligned}
    Z &= \frac{b^2 T^2}{(b^2-b^2T + T^2)^2}\\
    Z_\alpha &= \frac{(b^2-b(b-1)T)T^2}{(b^2-b^2T+T^2)^2}\\
    Z_d &= \frac{b T^3}{(b^2-b^2T + T^2)^2}\mathrlap{\quad(d\neq\alpha)}
  \end{aligned}
\end{equation}
Our main interest here is in the total mass:
\begin{proposition}\label{prop:numass}
  The measure $\nu$ has a finite total mass.  If $\alpha=\beta$ the
  generating functions for the masses of the $\sD^l$ are given by
  \eqref{eq:Zdd} and  if $\alpha\neq\beta$ they are given by \eqref{eq:Zde}.
  In both cases one obtains
  \begin{equation}
   \nu(\bsD) = b^2\qquad \nu(\bsD[d])= b\quad (\forall d \in\sD)
  \end{equation}
\end{proposition}
\begin{proof}
  The values given are obtained via $T=1$ in the respective generating series.
\end{proof}
\begin{remark}\label{rem:gf}
  Let $b=10$.  Baillie \cite[Section 12]{bailliearxiv} proposes to approximate
  the contribution to $K_1$ from integers with $l$ decimal digits by
  $\log(10)N(l)/(10^l - 10^{l-1})$, because $\log(10)\approx \sum_{l(n)=l}
  1/n$, and here $N(l)$ is the number of length $l$ integers contributing to
  $K_1$.  Boldly applying this estimate already at $l=2$ we get in total
  $\log(10)(1-1/10)^{-1}\sum_{1\leq d<10} Z_d(1) = \log(10)\frac{10}{9}9\times
  10 = 100\log(10)$.

  We can use the same tools to get rigorous but very crude bounds: the
  contribution to $K_1$ by length $l$ integers (for $l\geq2$) is greater than
  $N(l)/10^{l}$ and less than $N(l)/10^{l-1}$.  So $K_1$ is more than
  $\sum_{1\leq d<10} Z_d(1)=90$ and less than ten times that: $90<K_1<900$.

  Note also that if we replace $T$ by $bt$ in the various $Z_d$ we get the generating
  functions counting how many strings of a given length obey the condition
  (here to have one occurrence of $\alpha\beta$), and have a given leading
  digit.  Summing from $d=1$ to $d=b-1$ we then get the generating function
  counting how many \emph{positive integers} of a given length are obeying the
  condition.
\end{remark}

Building up on the previous remark let's explain how we can combine the idea
with exactly computed partial sums.  We get as illustrative case $b=2$,
$\alpha=\beta=0$, as considered in the introduction (which is actually one of
the very few cases amenable to concrete realization for $K_1$ sums...)

Let $n_0=0$ and $n_l$ be the count of numbers of length $l$ (in binary)
positive integers containing exactly one instance of $00$.  So $n_1=0$,
$n_2=0$, $n_3=1$.  According to \eqref{eq:Zdd} one has
\[
g(T) = \sum_{l=1}^\infty n_l2^{-l}T^l = \frac{2T^3}{(4-2T-T^2)^2}
                    = \frac{T^3}{8(1-\frac12T-\frac14T^2)^2} 
\]
The value at $T=1$ (which is in the convergence region) is $2$.  Let's call
$r_0=g(1)$ and more generally $r_l =  \sum_{k>l}^\infty n_k2^{-k}T^k$.  Then
\[
g(T) = \sum_{l=1}^\infty (r_{l-1}-r_l)T^l = r_0 T + r_1(T^2-T)+r_2(T^3-T^2)+\dots
\]
\[
\implies \sum_{l=0}^{\infty} r_l T^l = \frac{r_0 - g(T)}{1 -T} = \frac{\text{some polynom}}{(1-\frac12T-\frac14T^2)^2}
\]
\[
\implies r_l = (a + bl)x^l + (c +dl)y^l\qquad (x+y=\frac12, xy=-\frac14)
\]
Let $x=\frac{1+\sqrt{5}}{4}\approx\np{0.809}$ be the largest one.  So we see
that $r_{l+1}\sim x r_l$ (certainly $(a,b)\neq(0,0)$) which is what we hinted
at earlier.  This means that to obtain one more decimal digit, we need circa
all integers with up to $11$ extra binary digits, i.e. we must multiply our upper
limit of summation by more than $2000$!

Let's suppose we have computed by brute force $\sum' 1/m$ for all admissible
integers up to $N$ binary digits (included), obtaining a partial sum
$S_{N}=S(n<2^{N})$ of $K_1$.  We are certain that $S_{N} + r_{N} < K_1 <
S_{N} + 2 r_{N}$.  Indeed for any integer $m$ of length $k$, $2^{-k} <m^{-1}
\leq 2\cdot 2^{-k}$.

We can compute $r_{N}$ (whose value verified otherwise can serve as
double-check that at least we had the correct number of admissible integers of
each length) by many theoretical or practical means, in particular as $2 -
\sigma_{N}$ where $\sigma_{N}=\sum_{0\leq l\leq N}n_l2^{-l}$ was computed at
the same time we evaluated $S_{N}$ by ``brute force''.  The \cite[Section
12]{bailliearxiv} heuristics suggests to use $S_{N} + 2\log(2) r_{N}$ as
approximation.  In practice we observed that it is not so good heuristics once
$N$ is large enough (but practice is limited to the four feasible cases...).

Note that for the above sketched numerical implementation we only need to know
the total mass to obtain rigorous lower and upper bounds (but we should rather
also count in passing number of terms to compare with the generating series).
For $K_1$ sums this is always $b(b-1)$ as all $Z_d$ independently of whether
$d=\alpha$ or not have the same value $b$ at $T=1$.  The situation is slightly
different for $K_0$ sums and is left to the reader, see Proposition
\ref{prop:mumass} as starting point.

\section{Numbers, functions, and integrals}

Although we shall not make much use of it in the notations, let's define a
function $x:\bsD\to[0,1)$ via the formula
\begin{equation}
  x(X) = \frac{n(X)}{b^{|X|}}
\end{equation}
Recall that $n(X)$ is the integer $d_l b^{l-1} + \dots + d_1$ for $X=d_l\dots
d_1$ and $l=|X|$, and that $n(\emptyset)=0$, so $x(\emptyset)=0$. The image of
$x$ is the set of $b$-imal real numbers in $[0,1)$.  Any function $f$ on
$[0,1)$ can be pull-back to $\bsD$ to define $x^*f(X) = f(x(X))$.  But we will
rather push-forward through $x$ the measures of interest to the half-open
interval $[0,1)$ and consider functions on this interval as in our previous
works \cite{burnolkempner, burnolirwin}.

We use the symbols $\mu$ and $\nu$ also to refer to the push-forward measures
$x_*(\mu)$ and $x_*(\nu)$.  As these are finite discrete measures we can
integrate any bounded function on the interval against them (see
\cite{burnolirwin}).  We have a notational problem here that we also need
integrating against the various $x_*(\muld{d})$, and if we write the measures
as $\muld{d}$ the integration element will be a somewhat unclear
$d\muld{d}(x)$.  So we use rather the notations $\mu_d = x_*(\muld{d})$ and
$\nu_d = x_*(\nuld{d})$ and $\mu_d(\dx)$, $\mu(\dx)$, rather than $d\mu_d(x)$
and $d\mu(x)$.

We translate the counting Lemma \ref{lem:main} into an identity with
integrals:
\begin{lemma}\label{lem:int}
  Let $f$ be any bounded function on $[0,b)$. Then, for
  any digit $d\in\sD$:
  \begin{align}\label{eq:intmud}
      b\int_{[0,1)}f(bx)\mu_d(\dx) &= \int_{[0,1)} f(d+x)(\mu-\delta_{d,\alpha}\mu_\beta)(\dx)\\
\label{eq:intnud}
      b\int_{[0,1)}f(bx)\nu_d(\dx) &= 
   \int_{[0,1)} f(d+x)(\nu-\delta_{d,\alpha}\nu_\beta+\delta_{d,\alpha}\mu_\beta)(\dx)
  \end{align}
\end{lemma}
Note that we use the Kronecker delta in these formulas.
\begin{proof}
  On the left hand side of the first formula we have the infinite sum over all
  positive integers $l$ and all strings $X$ of length $l$ having $d$ as
  leading digit of $b f(b \,n(X)/b^l)\muld{d}(\{X\}) = b
  f(b\,x(X))\muld{d}(\{X\})$.  Here $\muld{d}(\{X\})$ is $b^{-|X|}$ if $X\in\cX$
  else it is zero.  We can restrict to $l>0$ because the none-string receives
  no weight from $\muld{d}$ (this does not mean that the real number
  $0$ receives no weight from $\mu_d$!).

  Similarly on the right hand side we have the infinite sum over all positive
  integers $l=l-1+1$ and all strings $Y$ of length $l-1$ (not $l$... we do
  want to take into account the none-string) of $f(d + b^{1-l}n(Y))$ weighted
  by some (non-negative) quantity $(\mu-\delta_{d,\alpha}\mu_\beta)(\{Y\})$.
  According to Lemma \ref{lem:main} the used weight is also
  $b\tau_*(\muld{d})(\{Y\})$.  Now the fiber $\tau^{-1}(\{Y\})$ in $\bsD_{l}$
  consists of the extensions $eY$ of $Y$ by a leading digit $e$.  The measure
  $\muld{d}$ may assign some positive weight to a single one element in this
  fiber, which is $X = dY$.  It may assign it zero weight but what is certain
  is that all other elements in the fiber have zero weight.  The value $f(b\,
  x(X))$ is also $f(d+n(Y)/b^{l-1})$.  So
\[
  bf(b\,x(X))\muld{d}(\{X\}) = f(d+x(Y))(\mu-\delta_{d,\alpha}\mu_\beta)(\{Y\})
\]
Hence the first integral formula.  The second integral formula receives the
analogous proof as corollary to Lemma \ref{lem:main}.
\end{proof}

The case $\alpha=\beta$ is more complex than $\alpha\neq\beta$.  We assume
from now on $\boxed{\alpha\neq\beta}$ and will return to the case of
$\alpha=\beta$ in the final chapter only.
\begin{lemma}\label{lem:intmu}
  Let $f$ be any bounded function on $[0,b)$. Suppose $\alpha\neq\beta$.
Then
  \begin{equation}\label{eq:intmu}
      b\int_{[0,1)}f(bx)\mu(\dx) 
      = b f(0) + \sum_{d\in\sD}\int_{[0,1)} f(d+x)\mu(\dx)
      - \frac1b \int_{[0,1)}f(\alpha+\frac{\beta+ x}b)\mu(\dx)
\end{equation}
And for any bounded function $g$ on $[0,1)$:
\begin{equation}\label{eq:intmubeta}
  \int_{[0,1)} g(x)\mu_\beta(\dx) = \frac1b\int_{[0,1)} g(\frac{\beta + x}b)\mu(\dx)
\end{equation}
  \begin{equation}\label{eq:intmug}
      \int_{[0,1)}g(x)\mu(\dx) 
      = g(0) + \frac1b\sum_{d\in\sD}\int_{[0,1)} g(\frac{d+x}b)\mu(\dx) -
     \frac1{b^2}
       \int_{[0,1)}g(\frac{\alpha b + \beta+ x}{b^2})\mu(\dx)
\end{equation}
\end{lemma}
\begin{proof}
  We start with the second equation.  Define $f$ on $[0,b)$ by the formula
  $f(x) = \frac1b g(\frac xb)$.  Then $g(x) = bf(bx)$ so, according to
  \eqref{eq:intmud} and to the hypothesis made here that $\beta\neq\alpha$:
\begin{equation}
  \int_{[0,1)} g(x)\mu_\beta(\dx) = \int_{[0,1)} f(\beta + x)\mu(\dx) =
  \frac1b \int_{[0,1)}g(\frac{\beta + x}b)\mu(\dx)
\end{equation}

  We make the sum of all equations \eqref{eq:intmud}, for all digits $d$.  We
  then use that $\sum_{d\in\sD} \mu_d = \mu - \delta_0$.  So we obtain at this stage
  \begin{equation}
       b\int_{[0,1)}f(bx)\mu(\dx) - b f(0) 
      = \sum_{d\in\sD}\int_{[0,1)} f(d+x)\mu(\dx) -\int_{[0,1)}f(\alpha + x)\mu_\beta(\dx)
  \end{equation}
  We combine with the integral formula \eqref{eq:intmubeta} and obtain
  \eqref{eq:intmu}.  The last equation \eqref{eq:intmug} is a reformulation of
  the latter, using again $g(x) = bf(bx)$.
\end{proof}

Let $g$ be defined as $g(x) = x^{-1}$ for $x\geq b^{-1}$ and $g(x) = 0$ for $0\leq x <
b^{-1}$.  Then the integral of $g$ against $\mu$ is the sum over all positive
integers $l$ and all strings $X\in \cX$ of length $l$ and non-zero leading
digit $\sLD(X)$ of the quantities $x(X)^{-1} b^{-|X|} = n(X)^{-1}$.  Hence
(see the Introduction for definition of $K_0$):
  \begin{equation}
    \label{eq:Kint1}
    \int_{[b^{-1}, 1)}\frac{\mu(\dx)}x = \sum_{X\in\cX,\,
      \sLD(X)\neq0}\frac1{n(X)} = K_0
  \end{equation}
The above reasoning works independently of whether $\alpha\neq\beta$ or not.
\begin{proposition}
  The value of the subsum $K_0$ of the harmonic series where only denominators
  having no occurrence of the string $\alpha\beta$ are kept is equal to
  $\int_{[b^{-1}, 1)}\frac{\mu(\dx)}x$.

If $\alpha\neq\beta$, it can be
  reformulated in succession as:
  \begin{align}
    \label{eq:Kint2}
    K_0 &= \sum_{n_1=1}^{b-1}\int_{[0,1)}\frac{\mu(\dx)}{n_1 + x} 
    - \delta_{\alpha\neq0}(\alpha)\int_{[0,1)}\frac{\mu(\dx)}{\alpha b + \beta + x}    \\
\label{eq:Kint3}
K_0 &=
\begin{aligned}[t]
  \sum_{n_1=1}^{b-1}\frac1{n_1} + \sum_{\substack{1\leq n_1<b\\0\leq n_2<b}} \int_{[0,1)}\frac{\mu(\dx)}{n_1b + n_2 + x} 
  -\sum_{n_1=1}^{b-1}\int_{[0,1)}\frac{\mu(\dx)}{n_1 b^2 + \alpha b + \beta + x}\\
  - \delta_{\alpha\neq0}(\alpha)\int_{[0,1)}\frac{\mu(\dx)}{\alpha b + \beta + x}
\end{aligned}
  \end{align}
\end{proposition}
\begin{proof}
  We explained \eqref{eq:Kint1} already.  For \eqref{eq:Kint2} we use
  \eqref{eq:intmug} with $ g(x) = x^{-1}$ for $x\geq b^{-1}$ and $g(x) = 0$
  for $0\leq x < b^{-1}$.  Then we apply again \eqref{eq:intmug} but using now
  the functions $(n_1+x)^{-1}$, $1\leq n_1<b$, and we obtain \eqref{eq:Kint3}.
\end{proof}
To analyse the combinatorial meaning of \eqref{eq:Kint3} we define for any
positive integer $n$ (whether or not it contains the $\alpha\beta$ string):
\begin{equation}\label{eq:un}
  U(n) = \int_{[0,1)} \frac{\mu(\dx)}{n+x}
\end{equation}
We may also use in intermediate steps the allied function
$U_\beta(n) =  \int_{[0,1)} \frac{\mu_{\beta}(\dx)}{n+x}$, but this is simply
according to equation \eqref{eq:intmubeta}
\begin{equation}
  U_\beta(n) = U(nb + \beta)
\end{equation}
Let's use the notation $\overline{d_1d_2\dots d_k}$ for the integer with these
digits $d_i$ (in this order).  We can reformulate \eqref{eq:Kint3} as:
\begin{equation}\label{eq:KU0}
  K_0 =   \sum_{n_1=1}^{b-1}\frac1{n_1} 
  + \sum_{b\leq \overline{n_1n_2}<b^2} U(\overline{n_1n_2})
  - \delta_{\alpha\neq0}(\alpha)U(\overline{\alpha\beta})
  -\sum_{n_1=1}^{b-1}U(\overline{n_1\alpha \beta})
\end{equation}
Let $l$ be the number of digits of $n>0$.  Arguing as for equation \eqref{eq:Kint1}:
\begin{equation}\label{eq:un_as_Ksum}
  U(n) = \sum_{\ld_l(m)=n,\text{ trailing part has no }\alpha\beta}\frac1m
\end{equation}
Here $\ld_l(m)$ is defined to be the integer constituted from the $l$ leading
digits of $m$.
By trailing part we mean the string of those digits of $m$ after the first $l$
ones.  Similarly
\begin{equation}
  U_\beta(n)
  \begin{aligned}[t]
   &= \sum_{\ld_{l}(m)=n,\text{ trailing part starts with
      }\beta\text{ and has no }\alpha\beta}\frac1m\\
   &= 
\sum_{\ld_{l+1}(m)=nb+\beta,\text{ trailing part has no }\alpha\beta}\frac1m
  \end{aligned}
\end{equation}
This confirms $U_\beta(n) = U(nb+\beta)$.

The integral identity \eqref{eq:intmug} applied to the fonction $g(x) = (n+x)^{-1}$
for $n$ a positive integer, gives
\begin{equation}
  \label{eq:Ueq}
  U(n) = \frac1n + \sum_{0\leq n_2<b} U(n b +n_2) - U(nb^2+\alpha b +\beta)
\end{equation}

Let's pause to comment on the combinatorics of \eqref{eq:KU0}.
In formula \eqref{eq:KU0} we first have the contribution of single-digit
numbers then we group terms in $K_0$ according to their first two digits $n_1$
and $n_2$.  Of course they are not allowed to be the $\alpha\beta$ length $2$
string.  So we must not incorporate $U(\overline{\alpha\beta})$.  But if
$\alpha=0$ there is nothing to take out as $n_1>0$ anyhow.  This explains the
term with $\delta_{\alpha\neq0}$.  As per the last term it is due to the
problem that when $n_2=\alpha$ we should not add digits to the right with
first one being $\beta$.  In $U(\overline{n_1\alpha})$ there will be a
contribution for this situation.  Hence we must now remove all $1/m$'s where
$m$ starts with digits $n_1\alpha\beta$ and then has a trailing string not
containing $\alpha\beta$.  This quantity is $U_{\beta}(\overline{n_1\alpha})$.
This explains all terms in \eqref{eq:KU0}.

We now turn to $K_1$ and study it via the measure $\nu$.  We need first the
generalization of the integral equations from Lemma \ref{lem:intmu}:
\begin{lemma}\label{lem:intnu}
  Let $f$ be any bounded function on $[0,b)$. Suppose $\alpha\neq\beta$.  Then
  \begin{equation}\label{eq:intnu}
    \begin{split}
      b\int_{[0,1)}f(bx)\nu(\dx) 
      = \sum_{d\in\sD}\int_{[0,1)} f(d+x)\nu(\dx) -
      \frac1b\int_{[0,1)}f(\alpha+\frac{\beta+ x}b)\nu(\dx)
      \\+\frac1b\int_{[0,1)}f(\alpha+\frac{\beta+ x}b)\mu(\dx)
    \end{split}
  \end{equation}
And for any bounded function $g$ on $[0,1)$:
\begin{equation}\label{eq:intnubeta}
  \int_{[0,1)} g(x)\nu_\beta(\dx) =
  \frac1b\int_{[0,1)} g(\frac{\beta + x}b)\nu(\dx)
\end{equation}
  \begin{equation}\label{eq:intnug}
    \begin{split}
      \int_{[0,1)}g(x)\nu(\dx) 
      = \frac1b\sum_{d\in\sD}\int_{[0,1)} g(\frac{d+x}b)\nu(\dx) -
      \frac1{b^2}\int_{[0,1)}g(\frac{\alpha b + \beta+ x}{b^2})\nu(\dx)
      \\+\frac1{b^2}\int_{[0,1)}g(\frac{\alpha b + \beta+ x}{b^2})\mu(\dx)
    \end{split}
  \end{equation}
\end{lemma}
\begin{proof}
  We start with equation \eqref{eq:intnubeta}.  Define $f$ on $[0,b)$ by the formula
  $f(x) = \frac1b g(\frac xb)$.  Then $g(x) = bf(bx)$ so, according to
  \eqref{eq:intnud} and to the fact that $\beta\neq\alpha$:
\[
  \int_{[0,1)} g(x)\nu_\beta(\dx) = \int_{[0,1)} f(\beta + x)\nu(\dx) =
  \frac1b \int_{[0,1)}g(\frac{\beta + x}b)\nu(\dx)
\]
  We make the sum of all equations \eqref{eq:intnud}, for all digits $d$.  We
  use that $\sum_{d\in\sD} \nu_d = \nu$.  So we obtain at this stage
  \begin{equation}
    \begin{split}
      b\int_{[0,1)}f(bx)\nu(\dx)
      = \sum_{d\in\sD}\int_{[0,1)} f(d+x)\nu(\dx) 
       -\int_{[0,1)}f(\alpha + x)\nu_\beta(\dx)
      \\+\int_{[0,1)}f(\alpha + x)\mu_\beta(\dx)
    \end{split}
  \end{equation}
  We combine with the integral formulas \eqref{eq:intnubeta} and
  \eqref{eq:intmubeta} and obtain \eqref{eq:intnu}.  The last equation
  \eqref{eq:intnug} is a reformulation of the latter, using again $g(x) =
  bf(bx)$.
\end{proof}

Let's take $g(x) = x^{-1}$ for $x\geq b^{-1}$ and $g(x) = 0$ for $0\leq x <
b^{-1}$.  Then the integral of $g$ against $\nu$ is the sum over all positive
integers $l$ and all strings $X\in \cY$ (i.e. strings with exactly one
occurrence of $\alpha\beta$) of length $l$ and non-zero leading digit of
$x(X)^{-1} b^{-|X|} = n(X)^{-1}$.  Hence:
  \begin{equation}
    \label{eq:Kone1}
    K_1 = \int_{[b^{-1}, 1)}\frac{\nu(\dx)}x
  \end{equation}
\begin{proposition}
  The value of the subsum $K_1$ of the harmonic series where only denominators
  having exactly one occurrence of the string $\alpha\beta$ are kept is equal
  to $\int_{[b^{-1}, 1)}\frac{\nu(\dx)}x$.  For $\alpha\neq\beta$, it can be
  reformulated in succession as:
  \begin{align}
    \label{eq:Kone2}
    K_1 &=
      \sum_{n=1}^{b-1}\int_{[0,1)}\frac{\nu(\dx)}{n + x} 
      - \delta_{\alpha\neq0}(\alpha)\int_{[0,1)}\frac{(\nu-\mu)(\dx)}{\alpha b + \beta + x}
\\
\label{eq:Kone3}
K_1 &=
\begin{aligned}[t]
  \sum_{\substack{1\leq n_1<b\\0\leq n_2<b}} \int_{[0,1)}\frac{\nu(\dx)}{n_1b + n_2 + x}
     &-\sum_{n_1=1}^{b-1}\int_{[0,1)}\frac{(\nu-\mu)(\dx)}{n_1 b^2 + \alpha b + \beta + x}\\
    &- \delta_{\alpha\neq0}(\alpha)\int_{[0,1)}\frac{(\nu-\mu)(\dx)}{\alpha b + \beta + x}
\end{aligned}
  \end{align}
\end{proposition}
\begin{proof}
  We explained \eqref{eq:Kone1} already.  For \eqref{eq:Kone2} we use
  \eqref{eq:intnug} with $ g(x) = x^{-1}$ for $x\geq b^{-1}$ and $g(x) = 0$
  for $0\leq x < b^{-1}$.  Then we apply again \eqref{eq:intnug} but using now
  the functions $(n_1+x)^{-1}$ and we obtain \eqref{eq:Kone3}.
\end{proof}
To analyse the combinatorial meaning of \eqref{eq:Kone3} we define
for positive integers $n$:
\begin{equation}\label{eq:vn}
  V(n) = \int_{[0,1)} \frac{\nu(\dx)}{n+x}
\end{equation}
Let $l$ be the number of digits of $n$.  Arguing as precedently:
\begin{equation}\label{eq:vn_as_Ksum}
  V(n) = \sum_{\ld_l(m)=n,\text{ trailing part has exactly one }\alpha\beta}\frac1m
\end{equation}
There is an identity which we deduce from using \eqref{eq:intnug} with $g(x) = (n+x)^{-1}$:
\begin{equation}
  \label{eq:Veq}
  V(n) = \sum_{0\leq n_1<b} V(nb + n_1) - V(nb^2+\alpha b+\beta) + U(nb^2+\alpha b+\beta)
\end{equation}

We will use in intermediate steps the auxiliary $V_\beta(n) = \int_{[0,1)}
\frac{\nu_{\beta}(\dx)}{n+x}$ which according to equation \eqref{eq:intnug} is
simply $V(nb+\beta)$.

Recall $\alpha\neq\beta$.  Comparing \eqref{eq:Kone3} with \eqref{eq:Kint3} we
see it appears provisorily advantageous to consider the sum $K_0 + K_1$:
\begin{equation}
  \begin{aligned}
    K_0 + K_1 
=  \sum_{n_1=1}^{b-1}\frac1{n_1} + \sum_{\substack{1\leq n_1<b\\0\leq n_2<b}} \int_{[0,1)}\frac{(\mu+\nu)(\dx)}{n_1b + n_2 + x}
    &-\sum_{n_1=1}^{b-1}\int_{[0,1)}\frac{\nu(\dx)}{n_1 b^2 + \alpha b + \beta + x}\\
    &- \delta_{\alpha\neq0}(\alpha)\int_{[0,1)}\frac{\nu(\dx)}{\alpha b + \beta + x}
  \end{aligned}
\end{equation}
This gives:
\begin{equation}
  \label{eq:Katmost}
  K_0 + K_1 = \sum_{n=1}^{b-1}\frac1{n} + \sum_{b\leq \overline{n_1n_2}<b^2}
  U(\overline{n_1n_2})
+  \sum_{b\leq \overline{n_1n_2}\neq \overline{\alpha\beta}<b^2} V(\overline{n_1n_2}) 
- \sum_{n_1=1}^{b-1}V_\beta(\overline{n_1\alpha})
\end{equation}
Explanation: we are summing $1/n$ for positive integers containing at most one
instance of $\alpha\beta$.  First comes the contribution of single-digit
integers.  Then we handle all integers $n$ with at least two digits and the
first two digits are $n_1n_2$, so $n=n_1n_2Z$ where $n_1\geq1$ and $Z$ is a
trailing string which has at most one occurrence of $\alpha\beta$.  Our count
is in excess: it includes contributions from integers having 2 instances of
$\alpha\beta$.  We must first remove the terms where $n_1n_2=\alpha\beta$ and
$Z$ contains exactly one occurrence of $\alpha\beta$.  The only other way to
have two occurrences is when $n_2=\alpha$, and $Z$ starts with $\beta$ and $Z$
has one occurrence of $\alpha\beta$.  We must take this out.
%
%

Combining \eqref{eq:KU0} with we deduce the last proposition of this section:
\begin{proposition}\label{prop:K}
  Let $K_i$, $i\in\NN$, be the sums of the reciprocals of integers having
  exactly $i$ occurrences of the string $\alpha\beta$.  Suppose
  $\boxed{\alpha\neq\beta}$.  Then $K_0$ and $K_1$ are expressed as
  combinations of the functions $U(n)$ and $V(n)$ defined in equations
  \eqref{eq:un} and \eqref{eq:vn} as Stieltjes transforms (up to sign) of
  certain measures $\mu$ and $\nu$ on the half-open interval $[0,1)$:
\begin{align}
  K_0 &= \sum_{n_1=1}^{b-1}\frac1{n_1}
  + \sum_{b\leq \overline{n_1n_2}\neq\overline{\alpha\beta}<b^2} U(\overline{n_1n_2})
   -\sum_{1\leq n_1<b}U(\overline{n_1\alpha\beta})
  \\[2\jot]
K_1 &=
\begin{aligned}[t]
  &\sum_{b\leq \overline{n_1n_2}\neq\overline{\alpha\beta}<b^2} V(\overline{n_1n_2})
  - \sum_{1\leq n_1< b} V(\overline{n_1\alpha\beta})\\
  &+ \delta_{\alpha\neq0} U(\overline{\alpha\beta})
  + \sum_{1\leq n_1< b} U(\overline{n_1\alpha\beta})
\end{aligned}
\end{align}
We used notations such as $\overline{n_1n_2}$ for the integer $n_1b+n_2$.  The
$U$ and $V$ functions obey the recursive identities \eqref{eq:Ueq} and
\eqref{eq:Veq} allowing to replace their arguments with arguments having more
digits.
\end{proposition}
\begin{remark}
  This proposition, \emph{except for the} fact that $U(n)$ and $V(n)$ are
  Stieltjes transforms can be established with no preparatory work immediately
  from \eqref{eq:un_as_Ksum} and \eqref{eq:vn_as_Ksum} in few lines of
  combinatorial explanations (that we have also included here).  The
  expression as integrals is what will convert ultimately
  the Proposition into a numerical algorithm.
\end{remark}

\section{Moments and their recurrences}

We define for all non-negative integers:
\begin{equation}
  u_m = \mu(x^m) = \int_{[0,1)} x^m \mu(\dx)\qquad
  v_m = \nu(x^m) = \int_{[0,1)} x^m \nu(\dx)
\end{equation}
We know $u_0 = b^2$.  From equation \eqref{eq:intmu} we get
for $m\geq1$,
\begin{equation}
  b^{m+1}u_m
    =\sum_{0\leq d<b}\sum_{j=0}^m\binom{m}{j}d^ju_{m-j}
     -b^{-m-1}\sum_{j=0}^m\binom{m}{j}(\alpha b + \beta)^j u_{m-j}
\end{equation}
Using the definition
\begin{equation}\label{eq:gammaj}
  \gamma_j =  \sum_{0\leq d < b} d^j\qquad (\gamma_0 = b)
\end{equation}
the recurrence takes the shape
\begin{equation}
\label{eq:recurru}
 (b^{m+1} - b + b^{-m-1}) u_m   = 
 \sum_{j=1}^{m}\binom{m}{j} \Bigl(\gamma_j - b^{-m-1}(\alpha b + \beta)^j\Bigr) u_{m-j}
\end{equation}
For example for $m=1$ this gives
\begin{equation}
  (b^2 -b + b^{-2})u_1 = \Bigl(\frac{b(b-1)}2 - b^{-2}(\alpha b + \beta)\Bigr) u_0
\end{equation}
We observe that $u_1 < \frac12 u_0$.  We will not investigate here the $(u_m)$
sequence asymptotics as we did in the analogous case of a one-digit excluded
string in \cite{burnolkempner}.  We know a priori from its definition as the
sequence of moments for the measure $\mu$ (which is supported on $[0,1)$ and not
reduced to a multiple of a Dirac at the origin) that it is a positive strictly
decreasing sequence converging to zero.

We can express $U(n)$ with these moments:
\begin{equation}
\label{eq:unseries}
  U(n) = \frac{b^2}{n} + \sum_{m=1}^\infty (-1)^m \frac{u_m}{n^{m+1}}
\end{equation}
Indeed it is only a matter to expand $(n+x)^{-1}$ in a series in powers of
$x/n$ and then integrate termwise.  We are using only integers with at least
two digits, but it would work also with $n=1$.

We can now give a geometrically converging alternating series in the spirit of
\cite{burnolkempner,burnolirwin} to compute the ``no $\alpha\beta$'' harmonic
series:
\begin{theorem}\label{thm:K0}
  Let $b>1$ and let $\alpha\neq \beta$ two \emph{distinct} digits in base $b$,
  and let $K_0$ be defined as formerly.  There holds
\begin{equation*}
\begin{split}
  K_0 = \sum_{n_1=1}^{b-1}\frac1{n_1}
  + b^2\Bigl(\sum_{\substack{1\leq n_1<b\\0\leq n_2<b\\(n_1,n_2)\neq(\alpha,\beta)}} \frac1{n_1b+n_2}
  -\sum_{1\leq n_1<b}\frac1{n_1b^2+\alpha b + \beta}\Bigr)
  \\+ \sum_{m=1}^\infty (-1)^m u_m\Bigl(
  \sum_{\substack{1\leq n_1<b\\0\leq n_2<b\\(n_1,n_2)\neq(\alpha,\beta)}} \frac1{(n_1 b + n_2)^{m+1}}
  -\sum_{1\leq n_1<b}\frac1{(n_1 b^2 + \alpha b + \beta)^{m+1}}\Bigr)
\end{split}
\end{equation*}
where the sequence $(u_m)$ obeys $u_0=b^2$ and equation \eqref{eq:recurru}.
\end{theorem}
\begin{proof}
  The formula is simply a combination of equation \eqref{eq:unseries} with
  Proposition \ref{prop:K}.
\end{proof}
\begin{remark}
  Notice that $(n_1b+\alpha)^{-m-1} - (n_1b^2+\alpha b+\beta)^{-m-1}>0$ and
  decreases when $m$ increases (left to reader).  And recall that the $u_m$
  are decreasing.  So we have here an alternating series with decreasing
  absolute values (even starting from $m=0$).
\end{remark}


Let's turn to the $v_m=\nu(x^m)$.  We know $v_0 = b^2$.  We use the integral
formula \eqref{eq:intnu} with $f(bx) = b^mx^m$:
\begin{equation}
  b^{m+1} v_m
  \begin{aligned}[t]
    ={}&\sum_{j=0}^m\binom{m}{j}\bigl(\sum_{0\leq d<b}d^j\bigr) v_{m-j}
      -b^{-m-1}\sum_{j=0}^m\binom{m}{j}(\alpha b + \beta)^j v_{m-j}\\
    &+ b^{-m-1}\sum_{j=0}^m\binom{m}{j}(\alpha b + \beta)^j u_{m-j}
  \end{aligned}
\end{equation}
Which can be reformulated as (see \eqref{eq:gammaj} for the $\gamma_j$'s):
\begin{equation}\label{eq:recurrv}
    (b^{m+1}  - b + b^{-m-1})v_m
  = 
  \begin{aligned}[t]
    &\sum_{j=1}^m\binom{m}{j}
    \begin{aligned}[t]
      \Bigl(&\bigl(\gamma_j- b^{-m-1}(\alpha b + \beta)^j\bigr)v_{m-j}\\
      &+ b^{-m-1}(\alpha b + \beta)^j u_{m-j}\Bigr)
    \end{aligned}
    \\
    &+b^{-m-1}u_m
  \end{aligned}
\end{equation}
From their definition as moments for the measure $\nu$, which is a discrete
and positive measure supported on $[0,1)$ and not reduced to a multiple of a
Dirac at the origin we know that the $v_m$'s build a positive decreasing and
converging to zero sequence.

We can express $V(n)$ as a power series:
\begin{equation}\label{eq:vnseries}
  V(n) = \frac{b^2}{n} + \sum_{m=1}^\infty (-1)^m \frac{v_m}{n^{m+1}}
\end{equation}
We can at last formulate the Theorem which motivated all the developments so
far in this paper:
\begin{theorem}\label{thm:K1}
  Let $b>1$ and $\alpha\neq \beta$ two \emph{distinct} digits in base $b$.
  Let $K_1$ be the sum of the reciprocals of integers having in base $b$
  exactly one occurrence of the string $\alpha\beta$.  There holds
\[
K_1 =
b^2\sum_{b\leq n<b^2}\frac1{n}
\begin{aligned}[t]
  &+ \sum_{m=1}^\infty (-1)^m v_m 
  \Biggl(\sum_{\substack{1\leq n_1<b\\ 0\leq n_2<b\\(n_1,n_2)\neq(\alpha,\beta)}}
  \frac1{\overline{n_1n_2}^{m+1}}
  - \sum_{1\leq n_1 <b}\frac1{\overline{n_1\alpha\beta}^{m+1}}
  \Biggr)\\
  &+ \sum_{m=1}^\infty (-1)^m u_m 
  \Biggl(\frac{\delta_{\alpha\neq0}}{\overline{\alpha\beta}^{m+1}}+ 
          \sum_{1\leq n_1 <b}\frac1{\overline{n_1\alpha\beta}^{m+1}}
   \Biggr)
\end{aligned}
\]
where the sequences $(u_m)$ and $(v_m)$ obey $u_0=b^2$, $v_0=b^2$ and the
recurrences \eqref{eq:recurru} respectively \eqref{eq:recurrv}.  Notations
such as $\overline{n_1n_2n_3}$ mean $n_1b^2+n_2b+n_3$.
\end{theorem}
\begin{proof}
  This is a direct corollary to Proposition \ref{prop:K} and the series
  expansion \eqref{eq:unseries} and \eqref{eq:vnseries} of the $U(n)$ and
  $V(n)$ quantities.  As $u_0 = v_0 = b^2$ we could conveniently regroup the
  zero-th moment contributions as a single sum over all integers with two
  digits (independently of whether $\alpha$ is zero or not).
\end{proof}

\section{The \texorpdfstring{$\alpha=\beta$}{alpha=beta} case}

We turn to the $\alpha=\beta$ case, which is more complex.  The main players
will not be the measures $\mu$ and $\nu$,
but (with notations here as measures on the interval $[0,1)$):
\begin{equation}
  \sigma = \mu-\mu_\alpha\qquad\tau = \nu-\nu_{\alpha}
\end{equation}
It is their moments which are now serving to define the $(u_m)$ and $(v_m)$ sequences.
\begin{align}
    \label{eq:lambdamoments}
  u_m &= \sigma(x^m) = \int_{[0,1)} x^m \sigma(\dx)\\
  v_m &= \tau(x^m) = \int_{[0,1)} x^m \tau(\dx)
\end{align}
According to Propositions \ref{prop:mumass} and \ref{prop:numass} there holds
$u_0 = \sigma([0,1)) = b^2$ and $v_0=\tau([0,1))=b(b-1)$.  We first consider
$\sigma$ and explain the route to $K_0$.

The basic integration Lemma \ref{lem:int} is available and equation
\eqref{eq:intmud} tells us that for any bounded function $f$ on $[0,b)$ one
has
\begin{equation}\label{eq:ddintmualpha}
  b\int_{[0,1)}f(bx)\mu_\alpha(\dx) = \int_{[0,1)} f(\alpha + x)\sigma(\dx)
\end{equation}
and for $d\neq\alpha$:
\begin{equation}
  b\int_{[0,1)}f(bx)\mu_d(\dx) = 
 \int_{[0,1)} f(d + x)\sigma(\dx) + \int_{[0,1)} f(d + x)\mu_\alpha(\dx)
\end{equation}
We sum these for all $d\neq\alpha$ and observe that
$\sum_{d\neq\alpha}\mu_d=\sigma-\delta_{\emptyset}$ so
\begin{align}
  b\int_{[0,1)}f(bx)\sigma(\dx) - bf(0) &=
  \int_{[0,1)} \sum_{d\neq\alpha}f(d + x)\sigma(\dx)
+ \int_{[0,1)} \sum_{d\neq\alpha}f(d + x)\mu_\alpha(\dx) \\
&= \int_{[0,1)} \sum_{d\neq\alpha}f(d + x)\sigma(\dx)
+ \frac1b \int_{[0,1)} \sum_{d\neq\alpha}f(d + \frac{\alpha + x}b)\sigma(\dx)\notag
\end{align}
We used \ref{eq:ddintmualpha} to reformulate the last term.  Consequently now
using $\mu$ (observe that the first summation is here over all digits):
\begin{equation}\label{eq:ddintmuf}
    b\int_{[0,1)}f(bx)\mu(\dx) - bf(0) = \int_{[0,1)} \sum_{0\leq d<b}f(d + x)\sigma(\dx)
+ \frac1b \int_{[0,1)} \sum_{\substack{0\leq d < b\\d\neq\alpha}}f(d + \frac{\alpha + x}b)\sigma(\dx)
\end{equation}
which can also be expressed with now a bounded function $g$ on $[0,1)$:
\begin{equation}
  \label{eq:ddintmug}
  \int_{[0,1)}g(x)\mu(\dx)
= g(0) + \frac1b \int_{[0,1)} \sum_{0\leq d<b}g(\frac{d + x}b)\sigma(\dx)
+ \frac1{b^2} \int_{[0,1)} 
               \sum_{\substack{0\leq d < b\\d\neq\alpha}}
                    g(\frac{d b + \alpha + x}{b^2})\sigma(\dx)
\end{equation}
We will also need the transformation formula against the $\sigma$ measure, if
we want to iterate:
\begin{equation}\label{eq:ddintsigmag}
    \int_{[0,1)}g(x)\sigma(\dx)
= g(0) + \frac1b \int_{[0,1)} \sum_{\substack{0\leq d < b\\d\neq\alpha}}
                                   g(\frac{d + x}b)\sigma(\dx)
+ \frac1{b^2} \int_{[0,1)} \sum_{\substack{0\leq d < b\\d\neq\alpha}}
                               g(\frac{d b + \alpha + x}{b^2})\sigma(\dx)    
\end{equation}
Arguing as precedently we have
\begin{equation}
  K_0 = \int_{[b^{-1},1)}\frac{\mu(\dx)}{x}
\end{equation}
So we apply \eqref{eq:ddintmug} to $g$ defines as $g(x) = x^{-1}$
for $b^{-1}\leq x < 1$ and zero elsewhere.  This gives;
  \begin{equation}
    \label{eq:Kddint2}
    K_0 = \sum_{n_1=1}^{b-1}\int_{[0,1)}\frac{\sigma(\dx)}{n_1 + x} 
    + \sum_{\substack{1\leq n_1<b\\n_1\neq\alpha}}
                        \int_{[0,1)}\frac{\sigma(\dx)}{n_1 b + \alpha + x}
  \end{equation}
  Let's for positive integer $n$ define the Stieltjes (up to sign) transforms:
\begin{equation}\label{eq:unddseries}
  U(n) = \int_{[0,1)}\frac{\sigma(\dx)}{n+x} 
= \frac{b^2}n + \sum_{m=1}^\infty (-1)^m\frac{u_m}{n^{m+1}}
\end{equation}
The same arguments as used earlier show that with $l=l(n)$ being the number of
digits of $n$ one has
\begin{equation}\label{eq:unddKsum}
  U(n) = \sum_{\substack{\ld_l(m)=n\\
               \text{trailing part has no }\alpha\alpha\\
               \text{and its leading digit is not }\alpha}}\frac1m
\end{equation}
From this or from the equation \eqref{eq:ddintsigmag} with $g(x) = (n+x)^{-1}$ we
get this functional recursive identity:
\begin{equation}
  \label{eq:ddUeq}
  U(n) = \frac1n 
      + \sum_{\substack{0\leq d<b\\d\neq \alpha}} \Bigl(U(nb+d) + U(nb^2+db+\alpha)\Bigr)
\end{equation}
 So the expression \eqref{eq:Kddint2} for $K_0$ becomes, using $\overline{d_1...d_k}$ notations for integers built from digits:
\begin{equation}
  K_0 =\sum_{1\leq n_1 <b}\Bigl(U(n_1) + \delta_{n_1\neq\alpha}U(\overline{n_1\alpha})\Bigr)
\end{equation}
and the identity \eqref{eq:ddUeq} gives
\begin{equation}
  K_0
  =\sum_{n_1=1}^{b-1}\frac1{n_1} + \sum_{\substack{1\leq n_1<b\\0\leq n_2<b\\n_2\neq\alpha}}
    \Bigl(U(\overline{n_1n_2}) + U(\overline{n_1n_2\alpha})\Bigr)
    + \delta_{\alpha\neq0}\sum_{n_1\neq\alpha} U(\overline{n_1\alpha})
\end{equation}
which we reformulate as
\begin{equation}
  K_0 =\sum_{n_1=1}^{b-1}\frac1{n_1} +
    \sum_{\substack{1\leq n_1<b\\0\leq n_2<b\\(n_1,n_2)\neq(\alpha,\alpha)}}
      U(\overline{n_1n_2})
    + \sum_{\substack{1\leq n_1<b\\0\leq n_2<b\\n_2\neq\alpha}}
     U(\overline{n_1n_2\alpha})
\end{equation}
Define now
\begin{equation}\label{eq:gammathetaprime}
  \gamma_j' = \sum_{\substack{0\leq d<b\\d\neq\alpha}} d^j\qquad 
  \theta_j' = \sum_{\substack{0\leq d<b\\d\neq\alpha}} (db+\alpha)^j
\end{equation}
Then using \eqref{eq:ddintsigmag} with $g(x)= x^m$, $m\geq1$, we get
\begin{equation}
  b^{m+1}u_m = \sum_{j=0}^m\binom{m}{j}\gamma_j'u_{m-j} 
             + b^{-m-1}\sum_{j=0}^m\binom{m}{j}\theta_j'u_{m-j}
\end{equation}
\begin{equation}
\label{eq:ddrecurru}
  (b^{m+1} - (b-1) - (b-1)b^{-m-1})u_m =
     \sum_{j=1}^m\binom{m}{j}(\gamma_j'+b^{-m-1}\theta_j')u_{m-j} 
\end{equation}
\begin{theorem}\label{thm:K0dd}
  Let $b>1$ and let $\alpha\in\{0,\dots,b-1\}$.  Let $K_0$ be the sum of the
  $1/n$'s for those integers having no $\alpha\alpha$ sub-string in their
  minimal representation with $b$-digits.  Then
\[
  \begin{split}
    K_0 = \sum_{n_1=1}^{b-1}\frac1{n_1}
    + b^2
    \Biggl(
    \sum_{\substack{1\leq n_1<b\\0\leq n_2<b\\(n_1,n_2)\neq(\alpha,\alpha)}}
    \frac1{n_1b+n_2} +\sum_{\substack{1\leq n_1<b\\0\leq n_2<b\\n_2\neq\alpha}}
    \frac{1}{n_1b^2+n_2b+\alpha}
    \Biggr)
    \\+\sum_{m=1}^\infty (-1)^m u_m
    \Biggl(
    \sum_{\substack{1\leq n_1<b\\0\leq n_2<b\\(n_1,n_2)\neq(\alpha,\alpha)}}
    \frac1{(n_1b+n_2)^{m+1}} +\sum_{\substack{1\leq n_1<b\\0\leq n_2<b\\n_2\neq\alpha}}
    \frac{1}{(n_1b^2+n_2b+\alpha)^{m+1}}
    \Biggr)
  \end{split}
\]
where $(u_m)$ is a positive sequence decreasing to zero and
verifying the recurrence \ref{eq:ddrecurru} with $u_0=b^2$.
\end{theorem}

Now we handle the $K_1$ sums.  We have for any bounded function $f$ on $[0,b)$
according to equation \eqref{eq:intnud}, and recalling that $\tau$ was defined
as $\nu-\nu_\alpha$:
\begin{equation}\label{eq:ddintnualpha}
  \begin{split}
    b\int_{[0,1)}f(bx)\nu_\alpha(\dx)
    = \int_{[0,1)} f(\alpha+x)(\nu-\nu_\alpha+\mu_\alpha)(\dx)
   \\ = \int_{[0,1]} f(\alpha+x)(\tau + \mu_\alpha)(\dx)
  \end{split}
\end{equation}
and for $d\neq\alpha$:
\begin{equation}
   b\int_{[0,1)}f(bx)\nu_d(\dx)
 = \int_{[0,1)} f(d+x)\nu(\dx)
 = \int_{[0,1]} f(d+x)(\tau + \nu_\alpha)(\dx)
\end{equation}
We sum these for all $d\neq\alpha$ and as
$\sum_{d\neq\alpha}\nu_d=\tau$ we obtain
\begin{gather}
  b\int_{[0,1)}f(bx)\tau(\dx)
=   \int_{[0,1)} \sum_{d\neq\alpha}f(d + x)\tau(\dx) 
+ \int_{[0,1)} \sum_{d\neq\alpha}f(d + x)\nu_\alpha(\dx) \\
= \int_{[0,1)} \sum_{d\neq\alpha}f(d + x)\tau(\dx)
+ \frac1b \int_{[0,1)} \sum_{d\neq\alpha}f(d + \frac{\alpha + x}b)(\tau+\mu_\alpha)(\dx)\\
\begin{aligned}[t]
  &= \int_{[0,1)} \sum_{d\neq\alpha}
     \Bigl(f(d + x)+\frac1bf(d + \frac{\alpha + x}b)\Bigr)\tau(\dx)\\
  &+\frac1{b^2}
    \int_{[0,1)} \sum_{d\neq\alpha}f(d + \frac{\alpha b + \alpha + x}{b^2})\sigma(\dx)
\end{aligned}
\end{gather}
We used \eqref{eq:ddintmualpha}.  Using it again in \eqref{eq:ddintnualpha} we get:
 for $f$ a function on $[0,b)$ 
\begin{equation}
  b\int_{[0,1)}f(bx)\nu_\alpha(\dx)
  = \int_{[0,1]} f(\alpha+x)\tau(\dx)
   + \frac1b\int_{[0,1]} f(\alpha+\frac{\alpha + x}{b})\sigma(\dx)
\end{equation}
Adding this to the previous equation gives
\begin{equation}\label{eq:ddnuf}
    b\int_{[0,1)}f(bx)\nu(\dx) = 
 \begin{aligned}[t]
  &\int_{[0,1)} \sum_{0\leq d<b}
     \Bigl(f(d + x) + \delta_{d\neq\alpha}\frac1bf(d + \frac{\alpha + x}b)\Bigr)\tau(\dx)\\
   &+\frac1{b}
    \int_{[0,1)} f(\alpha+\frac{\alpha + x}{b})\sigma(\dx)\\
 &+\frac1{b^2}
    \int_{[0,1)} \sum_{d\neq\alpha}f(d + \frac{\alpha b + \alpha + x}{b^2})\sigma(\dx)
\end{aligned}
\end{equation}
or equivalently for $g$ on $[0,1)$:
\begin{equation}\label{eq:ddnug}
    \int_{[0,1)}g(x)\nu(\dx) = 
\begin{aligned}[t]
  &\frac1b \int_{[0,1)} \sum_{0\leq d<b}
     \Bigl(g(\frac{d + x}b) +
     \delta_{d\neq\alpha}\frac1b g(\frac{d b + \alpha + x}{b^2})\Bigr)\tau(\dx)\\
   &+\frac1{b^2}
    \int_{[0,1)} g(\frac{\alpha b + \alpha + x}{b^2})\sigma(\dx)\\
 &+\frac1{b^3}
    \int_{[0,1)} \sum_{d\neq\alpha} 
      g(\frac{d b^2 + \alpha b + \alpha + x}{b^3})\sigma(\dx)
\end{aligned}
\end{equation}
So for any bounded function $g$ on $[0,1)$ we have:
\begin{equation}\label{eq:ddinttau}
    \int_{[0,1)}g(x)\tau(\dx) =
\begin{aligned}[t]
  &\frac1b \int_{[0,1)} \sum_{d\neq\alpha}
     \Bigl(g(\frac{d + x}b)+\frac1b g(\frac{d  b + \alpha + x}{b^2})\Bigr)\tau(\dx)\\
  &+\frac1{b^3}
    \int_{[0,1)} \sum_{d\neq\alpha}g(\frac{d b^2 + \alpha b + \alpha + x}{b^3})\sigma(\dx)
\end{aligned}
\end{equation}
Apply equation \eqref{eq:ddnug} with $g(x)$ being $x^{-1}$ if $b^{-1}\leq x <
1$ and zero elsewhere.  We know as previously that
\begin{equation}
  K_1 = \int_{[b^{-1},1)} g(x)\nu(\dx)
\end{equation}
So:
  \begin{equation}
    \label{eq:Koneddint2}
    K_1 =
    \begin{aligned}[t]
      & \sum_{n_1=1}^{b-1}\int_{[0,1)}\frac{\tau(\dx)}{n_1 + x} 
       + \sum_{\substack{n_1\neq\alpha\\n_1>0}}
      \int_{[0,1)}\frac{\tau(\dx)}{n_1 b + \alpha + x}\\
      &+  \delta_{\alpha\neq0}\int_{[0,1)}\frac{\sigma(\dx)}{\alpha b + \alpha + x}
       + \sum_{\substack{n_1\neq\alpha\\n_1>0}}
      \int_{[0,1)}\frac{\sigma(\dx)}{n_1 b^2 + \alpha b + \alpha + x}
    \end{aligned}
  \end{equation}
We define now:
\begin{equation}\label{eq:vnddseries}
  V(n) = \int_{[0,1)}\frac{\tau(\dx)}{n+x} 
= \frac{b(b-1)}n + \sum_{m=1}^\infty (-1)^m\frac{v_m}{n^{m+1}}
\end{equation}
Then we have with $l=l(n)$ being the number of digits of $n$:
\begin{equation}\label{eq:vnddKsum}
  V(n) = \sum_{\substack{\ld_l(m)=n\\
               \text{trailing part has exactly one }\alpha\alpha\\
               \text{and its leading digit is not }\alpha}}\frac1m
\end{equation}
And from equation \eqref{eq:ddinttau} we get the recursive property:
\begin{equation}\label{eq:ddVeq}
  V(n) =  \sum_{\substack{0\leq n_1\leq b\\n_1\neq\alpha}} \Bigl(V(nb+n_1)
                 + V(nb^2+n_1b+\alpha)
                 + U(nb^3+n_1b^2+\alpha b+ \alpha)\Bigr)
\end{equation}


We can rewrite our formula \eqref{eq:Koneddint2} for $K_1$ now as:
\begin{equation}
  K_1 = \sum_{0<n_1<b} V(n_1) 
    + \sum_{\substack{0<n_1<b\\n_1\neq\alpha}} V(\overline{n_1\alpha})
    + \delta_{\alpha\neq0} U(\overline{\alpha\alpha})
    + \sum_{\substack{0<n_1<b\\n_1\neq\alpha}} U(\overline{n_1\alpha\alpha})
\end{equation}
which is still a ``level $1$'' equation, so we apply the functional identity
\eqref{eq:ddVeq} to the first sum in order to reach level $2$ and get
after some rearrangement:
\begin{equation}\label{eq:ddK1}
  K_1 = \delta_{\alpha\neq0} U(\overline{\alpha\alpha}) 
       + \sum_{\substack{1\leq n_1<b\\0\leq n_2<b\\(n_1,n_2)\neq(\alpha,\alpha)}}
         V(\overline{n_1n_2})
  \begin{aligned}[t]
    & 
    + \sum_{\substack{1\leq n_1<b\\n_1\neq\alpha}} U(\overline{n_1\alpha\alpha})
   \\
   &+
    \sum_{\substack{1\leq n_1<b\\0\leq n_2<b\\n_2\neq\alpha}}
    \Bigl( 
     V(\overline{n_1n_2\alpha})
    +  U(\overline{n_1n_2\alpha\alpha})\Bigr)
  \end{aligned}
\end{equation}
From equation \eqref{eq:ddinttau} we get the recurrence of the moments
$v_m$. First, using therein a constant function we recover $v_0 = b(b-1)$ from
the fact that $u_0=\sigma([0,1))=b^2$.  More generally for all $m\geq0$:
\begin{equation}
  b^3 v_m = 
   \sum_{j=0}^m\binom{m}{j}\sum_{\substack{0\leq d<b\\d\neq\alpha}}
    \Bigl((b^{2-m}d^j
         + b^{1-2m} (db+\alpha)^j) v_{m-j}
         + b^{-3m}(db^2+\alpha b + \alpha)^j u_{m-j}\Bigr)
\end{equation}
Recall the definition \eqref{eq:gammathetaprime} of the quantities $\gamma_j'$ and
$\theta_j'$ for $j\geq0$.  We now also define
\begin{equation}
  \kappa_j' = \sum_{\substack{0\leq d<b\\d\neq\alpha}}(db^2+\alpha b + \alpha)^j
\end{equation}
and obtain the fundamental recurrence allowing to obtain the $(v_m)$ sequence
(and confirming $v_0 = b-1$):
\begin{equation}\label{eq:ddrecurrv}
  \begin{aligned}
    &(b^{m+1} - (b-1) - (b-1)b^{-m-1})v_m
\\
    &=
    \sum_{j=1}^m\binom{m}{j} \Bigl(
    (\gamma_j' + b^{-m-1} \theta_j')v_{m-j} + b^{-2m-2}\kappa_j' u_{m-j})
    \Bigr)
    + (b - 1) b^{-2m-2} u_m
  \end{aligned}
\end{equation}
Hence the theorem concluding this paper:
\begin{theorem}\label{thm:K1dd}
  Let $b>1$ and $\alpha\in\{0,\dots,b-1\}$.  Let $K_1$ be the sum of the
  reciprocals of integers having in base $b$ exactly one occurrence of the
  string $\alpha\alpha$.

  Let $(u_m)$ be the sequence verifying the recurrence \eqref{eq:ddrecurru}
  for $m\geq1$ and such that $u_0=b^2$ and let $(v_m)$ verify the recurrence
  \eqref{eq:ddrecurrv} for $m\geq0$.  These numbers are positive and decrease
  with limit zero.  Let $U$ and $V$ be the functions defined by the (inverse)
  power series:
\[ U(n) = \frac{b^2}{n} + \sum_{m=1}^\infty (-1)^m\frac{u_m}{n^{m+1}}
 \qquad 
V(n) = \frac{b(b-1)}{n} + \sum_{m=1}^\infty (-1)^m\frac{v_m}{n^{m+1}}
\]
Then $K_1$ is a finite sum of evaluations of the
functions $U$ and $V$ on certain integers having from $2$ to $4$ digits:
\begin{equation*}
  K_1 = \delta_{\alpha\neq0} U(\overline{\alpha\alpha}) 
    + \sum_{\substack{1\leq n_1<b\\0\leq n_2<b\\(n_1,n_2)\neq(\alpha,\alpha)}} 
      V(\overline{n_1n_2})
  \begin{aligned}[t]
    & 
    + \sum_{\substack{1\leq n_1<b\\n_1\neq\alpha}} U(\overline{n_1\alpha\alpha})
   \\
   &+
    \sum_{\substack{1\leq n_1<b\\0\leq n_2<b\\n_2\neq\alpha}}
    \Bigl( 
     V(\overline{n_1n_2\alpha})
    +  U(\overline{n_1n_2\alpha\alpha})\Bigr)
  \end{aligned}
\end{equation*}
\end{theorem}

\section{Acknowledgements}

The author thanks greatly Robert Baillie for drawing his attention to the
challenge of numerically evaluating the $K_i$ sums, which is the problem
solved here for $i=1$ and in a new way for $i=0$, and for our fruitful
exchanges on this topic.


\begin{thebibliography}{1}



\bibitem{schmelzerbaillie}
Thomas Schmelzer and Robert Baillie.
\newblock Summing a curious, slowly convergent series.
\newblock {\em Amer. Math. Monthly}, 115(6):525--540, 2008.
\newblock \url{https://doi.org/10.1080/00029890.2008.11920559}.

\bibitem{bailliearxiv}
Robert Baillie.
\newblock Summing the curious series of {K}empner and {I}rwin.
\newblock {\em arXiv}, 2008--2024 (v1--v6), arXiv:0806.4410.
\newblock \url{https://arxiv.org/abs/0806.4410}.

\bibitem{irwin}
Frank Irwin.
\newblock A {C}urious {C}onvergent {S}eries.
\newblock {\em Amer. Math. Monthly}, 23(5):149--152, 1916.
\newblock \url{https://doi.org/10.2307/2974352}.

\bibitem{burnolkempner}
Jean-Fran{\c c}ois Burnol.
\newblock Moments in the exact summation of the curious series of {K}empner
  type.
\newblock {\em Feb 4, 2024}, pages 1--10, 2024, arxiv:2402.08525.
\newblock \url{https://arxiv.org/abs/2402.08525}.

\bibitem{burnolirwin}
Jean-Fran{\c c}ois Burnol.
\newblock Moments in the summation of {I}rwin series.
\newblock {\em Feb 11, 2024}, pages
1--17, 2024, arxiv:2402.09083.
\newblock \url{http://arxiv.org/abs/2402.09083}.

\bibitem{fahri}
Bakir Farhi.
\newblock A curious result related to {K}empner's series.
\newblock {\em Amer. Math. Monthly}, 115(10):933--938, 2008.
\newblock \url{https://doi.org/10.1080/00029890.2008.11920611}.

\bibitem{borweinborwein}
Jonathan~M. Borwein and Peter~B. Borwein.
\newblock {\em Pi and the {AGM}}.
\newblock Canadian Mathematical Society Series of Monographs and Advanced
  Texts. John Wiley \& Sons, Inc., New York, 1987.
\newblock A study in analytic number theory and computational complexity, A
  Wiley-Interscience Publication.


\end{thebibliography}
\end{document}